\documentclass[10pt, amsmath, amssymb]{amsart}

\usepackage{amsmath}
\usepackage{amsthm}
\usepackage{amssymb}
\usepackage{graphics}
\usepackage{bm}

\usepackage{bbold}
\usepackage{tkz-linknodes}

\newcommand{\be}{\begin{equation}}
\newcommand{\ee}{\end{equation}}
\newcommand{\ba}{\begin{eqnarray}}
\newcommand{\ea}{\end{eqnarray}}
\newcommand{\bi}{\begin{itemize}}
\newcommand{\ei}{\end{itemize}}
\newcommand{\bn}{\begin{enumerate}}
\newcommand{\en}{\end{enumerate}}
\newcommand{\bp}{\begin{proof}}
\newcommand{\ep}{\end{proof}}

\newcommand{\wt}{\ensuremath{\widetilde}}

\newcommand{\mr}{\ensuremath{\mathrm}}

\newcommand{\mc}{\ensuremath{\mathcal}}
\newcommand{\mf}{\ensuremath{\mathfrak}}
\newcommand{\ov}{\ensuremath{\overline}}
\newcommand{\sm}{\ensuremath{\setminus}}

\newcommand{\intfty}{\ensuremath{\int _{-\infty} ^{\infty}} }
\newcommand{\Om}{\ensuremath{\Omega}}

\renewcommand{\bm}{\ensuremath{\mathbb }}

\newcommand{\vnm}{\mr{vN} (M) }
\newcommand{\id}{\mr{id}  }

\newcommand{\ip}[2]{\ensuremath{\langle {#1} , {#2} \rangle}}

\newcommand{\dom}[1]{\ensuremath{\mathrm{Dom} ({#1}) }}

\newcommand{\ran}[1]{\ensuremath{\mathrm{Ran} ({#1}) }}
\renewcommand{\ker}[1]{\ensuremath{\mathrm{Ker} ({#1}) }}

\newcommand{\im}[1]{\ensuremath{\mathrm{Im} \left( {#1} \right) }}
\newcommand{\re}[1]{\ensuremath{\mathrm{Re} \left( {#1} \right) }}
\newcommand{\ad}[1]{\ensuremath{\mathrm{Ad} _{#1} }}
\newcommand{\symr}[1]{\ensuremath{\mc{S}\mr{ym} ^R _1 ({#1}) }}
\newcommand{\sym}[1]{\ensuremath{\mc{S}\mr{ym} _1 ({#1}) }}

\numberwithin{equation}{section}

\newtheorem{thm}[subsection]{Theorem}

\newtheorem{lemming}[subsection]{Lemma}
\newtheorem{prop}[subsection]{Proposition}
\newtheorem{cor}[subsection]{Corollary}

-.9truein

\advance\voffset by -.1truein 

\begin{document}

\title{Near invariance and symmetric operators}

\author{R.T.W. Martin}

\address{Department of Mathematics and Applied Mathematics \\ University of Cape Town\\
Cape Town, South Africa \\
phone: +27 21 650 5734 \\ fax: +27 21 650 2334}

\email{rtwmartin@gmail.com}

\begin{abstract}

    Let $S$ be a subspace of $L^2 (\bm{R} )$. We show that the operator $M$ of multiplication by
the independent variable has a simple symmetric regular restriction to $S$ with deficiency indices $(1,1)$ if and only if $S = u h K^2 _\theta$ is a nearly invariant subspace, with $\theta$ a meromorphic inner function vanishing at $i$. Here $u$ is unimodular, $h$ is an isometric multiplier of $K^2 _\theta$ into $H^2$ and $H^2$ is the Hardy space of the upper half plane. Our proof uses the dilation theory of completely positive maps.

\vspace{5mm}   \noindent {\it Key words and phrases}: symmetric operators,  Hardy spaces, model subspaces, nearly invariant.

\vspace{3mm}

\end{abstract}

\maketitle

\section{Introduction}

A closed subspace $S \subset H^2 (\bm{U} )$, where $\bm{U}$ denotes the upper half plane is called nearly invariant \cite[Section 12]{Ross}, \cite{Sarason-near,Hitt} if the following condition holds:
\be f \in S \ \mr{and} \ f(i) =0 \ \ \Rightarrow \ \ \frac{f(z)}{z-i}  \in S .\ee  In other words the backwards shift (the adjoint of the restriction of multiplication by $\frac{z-i}{z+i}$ to $H^2$) maps the subspace $S' := \{ f \in S | \ \ f(i) = 0 \} \subset S$ into $S$. Any model subspace $K^2 _\theta$ is nearly invariant since it is by definition invariant for the backwards shift. Any nearly invariant
subspace of $H^2 (\bm{U} )$ can be written as $S = h K^2 _\theta $ where $\theta $ is inner, $\theta (i) =0$, and $h$ is a certain function such that $\frac{h(z)}{z+i} \in S$. A subspace $S \subset L^2 (\bm{R})$ is said to be
nearly invariant if $S = u S'$ where $u$ is a unimodular function and $S' \subset H^2$ is nearly invariant.

If $\theta$ is meromorphic, it is not difficult to show that any nearly invariant subspace $S = uh K^2 _\theta  \subset L^2 (\bm{R})$ is a reproducing kernel Hilbert space (RKHS) of
functions on $\bm{R}$ with a $\bm{T}$-parameter family of total orthogonal sets of point evaluation vectors.  This follows, for example, from the results of \cite{Martin-dB, Martin-symsamp} (these results show that any $K^2 _\theta$ has these properties for meromorphic inner $\theta$). It also follows that there is a linear manifold (non-closed subspace) $\dom{M_S} \subset S$ such that $M_S := M | _{\dom{M_S}}$ is a closed, regular and simple symmetric linear transformation with deficiency indices $(1,1)$.  Note that $M_S$ may not be densely defined, but the co-dimensions of its domain and range are at most $1$. We will denote the family of all such linear transformations on a Hilbert space $\mc{H}$ by $\symr{\mc{H}}$ for brevity. Here the $R$ stands for regular. Similarly let $\sym{S}$ denote the family of all simple symmetric linear transformations with deficiency indices $(1,1)$ that are defined in $S$.

The goal of this paper is to show that the two conditions: (i) $S$ is nearly invariant with $S= u h K^2 _\theta$ for meromorphic $\theta$ with $\theta (i) =0$ and (ii) $M$ has a symmetric restriction $M_S \in \symr{S}$, are in fact equivalent.  This will show in particular that the latter condition implies that $S$ is a RKHS with a $\bm{T}$-parameter family of total orthogonal sets of point evaluation vectors. One direction of $\mr{(i)} \  \Leftrightarrow \ \mr{(ii)}$ follows from known results - it is easy to show that if $S$ is nearly invariant, that $M$ has a symmetric restriction $M_S \in \sym{S}$ (in the next subsection we will show this follows from \emph{e.g.} \cite{Martin-dB}).  Proving the converse appears to be more difficult, and the goal of this paper is to accomplish this for the special case where $\theta$ has a meromorphic extension to $\bm{C}$.  In fact we expect that the more general result holds for arbitrary inner $\theta$.  That is, we conjecture that $S$ is nearly invariant if
and only if the multiplication operator $M$ has a simple symmetric restriction $M_S$ to a linear manifold in $S$ such that the Lifschitz characteristic function \cite{Lifschitz} of $M_S$ is inner (see also \cite[Appendix 1, Section 5]{Glazman}). Our approach to proving this result, however, would require the extension of several results in Krein's representation theory of simple symmetric operators to the non-regular case \cite{Krein}.  We will discuss this in more detail in the final section.

Given any symmetric operator $T \in \symr{\mc{H}}$ the results of \cite{Silva-entire,Silva-nondense} essentially show how to construct an isometry $V: \mc{H} \rightarrow L^2 (\bm{R})$ such that $\ran{V} = u K^2 _\theta$ for a meromorphic inner $\theta$ and $VTV^* = M_\theta$ acts as multiplication by the independent variable on its domain. They accomplish this by modifying and extending Krein's original representation theory for regular symmetric operators as presented in \cite{Krein}. Using this result, the theory of \cite{Krein}, and some dilation theory (Stinespring's dilation theorem for completely positive maps) we show that if $M$ has a symmetric restriction belonging to $\symr{S}$ where $S \subset L^2 (\bm{R})$, that $S=uh K^2 _\theta$ must be nearly invariant with meromorphic inner $\theta$ such that $\theta (i) =0$.  This provides another connection between the classical theory of representations of symmetric operators as originated by Krein and the theory of model subspaces of Hardy space.

\subsection{Nearly invariant subspaces of $H^2 (\bm{U} )$.}

Although it will be most convenient to work with the upper half-plane, nearly invariant subspaces of $H^2 (\bm{D})$ have a more elegant description. A subspace $S \subset H^2 (\bm{D})$ is called nearly invariant if  the following condition holds:
\be f \in S \ \ \mr{and} \ f(0) = 0  \ \ \Rightarrow \ f(z) /z  \in S.\ee  If a subspace $S \subset H^2 (\bm{D} )$ is nearly invariant then $S = h K^2 _\theta$ where $\theta$ is inner with $\varphi(0) = 0$, multiplication
by $h \in S$ is an isometry of $K^2 _\theta$ onto $S$, and $h$ is the unique solution to the extremal problem \cite{Hitt}:
\be \sup \{ \re{h(0)} | \ h \in S \ \mr{and} \ \| h \| =1  \} .\ee

Note that $h \in H^2$ since $\varphi(0) = 0$ implies that $k_0 ^\varphi(z) =1 \in K^2 _\theta$ is the point evaluation vector at $0$.
Conversely if $h$ is any isometric multiplier of $K^2 _\varphi$ into $H^2$ where $\varphi(0) =0$, then $S = h K^2 _\theta$ is nearly invariant
with extremal function $h$, and $h$ must have the form \cite{Sarason}:
\be h = \frac{a}{1-b \varphi} ,\ee where $a, b$ belong to the unit ball of $H^\infty$ and obey $|a| ^2 + |b | ^2 =1  \ \mr{a.e.}$ on the unit circle $\bm{T}$.

Nearly invariant subspaces of $H^2 (\bm{U})$ have a similar description as follows.
Let $\mu (z) := \frac{z-i}{z+i}$, $\mu : \ov{\bm{U} } \rightarrow \ov{ \bm{D} } \setminus \{ 1 \}$, which has compositional inverse $\mu ^{-1} (z) = i \frac{1+z}{1-z}$. Then $\mc{U} : H^2 (\bm{D} ) \rightarrow H^2 (\bm{U} )$
defined by \be \mc{U} f (z) := \frac{1 -\mu (z)}{\sqrt{\pi}} (f \circ \mu) (z) ,\ee

is a unitary transformation which maps $K^2 _\varphi \subset H^2 (\bm{D})$ onto $K^2 _{\varphi \circ \mu } \subset H^2 (\bm{U})$. If $S \subset H^2 (\bm{U} )$ is nearly invariant,
it follows that $S' := \mc{U} ^* S $ is nearly invariant and hence $S' = h K^2 _\varphi$ for some inner $\varphi \in H^\infty (\bm{D})$ such that $\varphi (0) =0$ and $h \in H^2 (\bm{D})$. It follows that $S = \mc{U} S' = (h \circ \mu ) K^2 _{\varphi \circ \mu } $ where $\mc{U} h =  \pi ^{-1/2} (1 -\mu  ) h \circ \mu \in S$, so that $\frac{h\circ \mu}{z+i} \in S \subset H^2 (\bm{U}) $. This shows that if $h'$ is any isometric multiplier of $K^2 _\theta$ into $H^2 (\bm{U} )$ (where $\theta (i) =0$), that $\frac{h'}{z+i} \in H^2$.

Given any inner function $\theta \in H^\infty (\bm{U} )$, it is well known that $M$ has a restriction $M_\theta \in \sym{K^2 _\theta}$ (see \emph{e.g} \cite{Martin-dB, Martin-symsamp}).  Suppose $S := h K^2 _\theta $ is nearly invariant ($\theta (i) =0$) and $h$ is an isometric multiplier of $K^2 _\theta$. Since $V:=$multiplication by $h$ commutes with $M$ and is an isometry of $K^2 _\theta $ onto $S$, it is not hard to see that $M_S =P_S V M _\theta V^* P_S$
is a symmetric restriction of $M$ to $S$ with domain $\dom{M _S} = V \dom{M _\theta}$.  Moreover, since $V \ran{M _\theta \pm _i } = \ran{M_S \pm i } $, it follows that $M_S \in \sym{S}$, and that the Lisvic characteristic function of $M_S$ is $\theta$ (recall here that $\theta (i) =0$).  This shows that any nearly invariant subspace has the property that $M$ has a restriction $M_S \in \sym{S}$.  The main goal of this paper is to show the converse (in the special case where $\theta$ is meromorphic),
namely that if $S \subset L^2 (\bm{R} )$ is such that $M_S \in \symr{S}$, that $S = uh K^2 _\theta$ is nearly invariant.

\section{Representation theory for symmetric operators}
\label{section:reptheory}

Let $\mc{H}$ be a separable Hilbert space and let $\sym{\mc{H}}$ denote the family of all closed simple symmetric linear transformations in $\mc{H}$ with deficiency indices $(1,1)$. By a linear transformation we mean a linear
map which is not necessarily densely defined, we reserve the term operator for a densely defined linear map. Even though it may not be densely defined, if $T \in \sym{\mc{H}}$, the co-dimensions of its domain and of its range are both equal to $n$ where $n$ is either $0$ or $1$. Notice that $\sym{\mc{H}} \supset \symr{\mc{H}}$.

Choose $\psi (i) \in \ran{T+i} ^{\perp}$ ($= \ker{T^* -i}$ in the case where $T$ is densely defined), and define the vector-valued function
\be \psi (z) := ( T' -i ) ( T' -z) ^{-1} \psi (i) = \psi (i) + (z-i) (T' -z) ^{-1} \psi (i) ,\ee
where $T'$ is any densely defined self-adjoint extension of $T$ within $\mc{H}$. If $T$ is regular then $T'$ has purely point spectrum consisting of eigenvalues of multiplicity one with no finite accumulation point, and it follows that $\psi (z)$ is meromorphic in $\bm{C}$, with simple poles at each point in $\sigma (T') \subset \bm{R}$. Also it can be shown that $0 \neq \psi (z) \in \ran{T-\ov{z} } ^\perp $ for all $z \in \bm{C} \sm \bm{R}$,
see \emph{e.g.} \cite[Section 1.2, pgs 8-9]{Krein}.

Choose $0 \neq u \in \ran{T + i} ^{\perp}$. One can establish the following:

\begin{lemming}
If $T \in \symr{\mc{H}}$ and $z \in \ov{\bm{U}}$, then for any non-zero $\psi _z \in \ran{T - \ov{z}} ^{\perp}$, $\ip{\psi _i }{ \psi _z} \neq 0$ \\
(so that $\ip{u}{\psi _z} \neq 0$).

\label{lemming:nopole}
\end{lemming}

The above lemma is a consequence of the following considerations:

Recall that $w \in \bm{C}$ is called a regular point of $T$ if $T-w $ is bounded below. Let $\Om $ denote the intersection of $\bm{U}$ with the set of all
regular points of $T$. Then $\bm{U} \subset \Om \subset \ov{\bm{U}}$ and $\Om = \ov{\bm{U}}$ if and only if $T$ is regular, \emph{i.e.} if and only if
$T \in \symr{\mc{H}}$.

Now for any $w \in \Om$, $\ran{T-\ov{w}} ^\perp = \bm{C} \{ \phi _w \}$ is one dimensional, spanned by a fixed non-zero vector $\phi _w$. For each $w \in \Om$,
let $\mf{D} _w := \dom{T} + \bm{C} \{ \phi _w \}$, and define the linear transformation $T_w$ with domain $\mf{D} _w$ by
\be T_w ( \phi + c \phi _w ) = T\phi + w c \phi _w , \ee for any $\phi \in \dom{T}$ and $c \in \bm{C}$. It is not difficult to verify that $T_w$ is a well-defined
and closed linear extension of $T$. Clearly $T_w$ is densely defined if $T$ is, in which case $T\subset T_w \subset T^*$. A quick calculation verifies that
$i T_w $ is dissipative, \emph{i.e.}  $\im{\ip{T_w \phi}{\phi} } \geq 0$ for all $\phi \in \mf{D} _w$. It follows from this that $T_w -z $ is bounded below for
all $z \in \bm{L}$, so that one can define $(T_w -z) ^{-1}$ as a linear transformation from $\ran{T_w -z}$ onto $\dom{T_w} = \mf{D} _w$. Observe that
$\phi _w$ is an eigenvector of $T_w$ to eigenvalue $w$ by construction.

\subsubsection{Remark} More can be said about the extensions $T_w$. Since we will not have need of these facts, we will state them here without proof.
If $T$ is not densely defined, then one can show that there is exactly one proper closed linear extension $T'$ of $T$ which is not densely defined, and this extension must
be self-adjoint. The transformations $T_w$ are self-adjoint if and only if $w \in \bm{R}$. (If $T_x$ is the self-adjoint extension of $T$ which
is not densely defined, it is self-adjoint in the sense of a linear relation, \emph{i.e.} its graph is self-adjoint as a subspace of $\mc{H} \oplus \mc{H}$
\cite{Snoo}). One can show that if $T_w$ is densely defined that $\sigma (T_w ) \subset \ov{\bm{U}}$. Since $iT_w$ is dissipative, it follows that
the Cayley transform $\mu(T_w)$ is a contractive linear operator which extends the isometric linear transformation $\mu (T)$. One can further show
that $w \in \Om$ is an eigenvalue of multiplicity one for $T_w$, and that $w \in \Om$ is an eigenvalue for both $T_w$ and $T_z$ if and only if $T_w = T_z$.

\begin{proof}{ (of Lemma \ref{lemming:nopole})}
    Choose $w =i \in \bm{U}$, and recall that $u \in \ran{T +i} ^\perp$. Suppose that $z \in \Om$. Then there is an extension $T_z$ of $T$ for which $\psi _z$ is an
eigenvector with eigenvalue $z$ (as described above).

If it were true that $\ip{u}{\psi _z} =0$ then we would have that $\psi _z \in \ran{T +i }$ so that $\psi _z = (T +i ) \phi$ for some $\phi \in \dom{T}$.
But then since $T_z -w $ is bounded below for all $w \in \bm{L}$ it would follow that $(z +i ) ^{-1} \psi _z = (T_z +i ) ^{-1} \psi _z = \phi   $ so that $\psi _z \in \dom{T}$.
This contradicts the fact that $T$ is simple (it also contradicts the fact that $T$ is symmetric if $z \notin \bm{R}$).
\end{proof}

It follows that the function $\ip{u}{\psi (\ov{z})}$ is meromorphic on $\bm{C}$ with zeroes contained strictly in the lower half-plane.

Now we can define the vector-valued function $\delta (z) := \frac{\psi (z)}{ \ip{\psi(z)}{u}}$.  By the previous lemma, this is meromorphic in $\bm{C}$ with
poles contained in the lower half-plane (the poles of $\psi (z)$ on $\bm{R}$ cancel out with those of $\ip{\psi(z)}{u}$, see \emph{e.g.} \cite{Silva}).

Hence one can define a linear map $V$ of $\mc{H}$ into a vector space of functions analytic on an open neighbourhood of the closed upper half-plane
by $(Vf) (z) := \ip{f}{\delta (\ov{z})} =: \hat{f} (z)$ for any $f \in \mc{H}$.  We can endow the range of $V$, $V\mc{H} =: \hat{\mc{H}}$
with an inner product which makes it a Hilbert space (and $V : \mc{H} \rightarrow \hat{\mc{H}}$ an isometry) as follows.

Let $Q$ denote any unital $B(\mc{H})$-valued POVM (Positive Operator Valued Measure) which diagonalizes $T$. In this case $Q(\Om ) = P \chi _{\Om} (S) P $ where $S$ is a self-adjoint extension
of $T$ (to perhaps a larger Hilbert space $\mc{K} \supset \mc{H}$), and $P : \mc{K} \rightarrow \mc{H}$ is orthogonal projection. Here we assume that $Q(\bm{R}) = \bm{1}$ so that
$S$ is a densely defined linear operator in $\mc{K}$ (this is always the case if $T$ is densely defined). Also here,
$\Om \in \mr{Bor} (\bm{R} ) :=$ the Borel sigma algebra of subsets of $\bm{R}$. The Borel measure defined by $\sigma (\Om ) := \ip{Q (\Om) u}{u} = \ip{ \chi _\Om (S) u}{u}$ is called a
$u$-spectral measure for $T$, and we have the following theorem \cite[Theorem 2.1.2, pg. 51]{Krein}:

\begin{thm}{(Krein)}
The map $Vf = \hat{f}$ is an isometric map of $\mc{H}$ into $L^2 (\bm{R} , d \sigma )$. It is onto if and only if $Q$ is a projection-valued measure (PVM).
\label{thm:transform}
\end{thm}

It is not hard to check that $V T V^* = \hat{T}$ acts as multiplication by the independent variable in $\hat{\mc{H}}$.

Silva and Toloza modify this construction slightly as follows \cite{Silva-entire}. Let $h(z)$ be any entire function whose zero set is equal to $\sigma (T')$ (such an entire function
always exists, since the spectrum of $\sigma (T')$ is a discrete set of real eigenvalues of multiplicity one with no finite accumulation point).  Then
define $\gamma (z) := h(z) \psi (z)$. Then they define the linear map $\wt{V}f (z) := \wt{f} (z) := \ip{f}{\gamma (\ov{z} )}$, which maps elements of $\mc{H}$
into a vector space $\wt{\mc{H}}$ of entire functions.  If one endows $\wt{\mc{H}}$ with
the inner product $\ip{\wt{f}}{\wt{g}}_{\wt{\mc{H}}} = \ip{f}{g}$, then $\wt{\mc{H}}$ is a Hilbert space, $\wt{V}$ is an isometry, and one can further verify
that $\wt{H}$ is actually an axiomatic de Branges space of entire functions.
It follows from results of de Branges that there is an entire de Branges function $E$ (which we can assume has no real zeroes by de Branges \cite[Problem 44, pg. 52]{dB}) such that $\wt{\mc{H}}$ with the
inner product $\ip{\wt{f}}{\wt{g}} _E := \intfty \wt{f} (x) \ov{\wt{g} (x)} \frac{1}{|E(x)|^{-2}} dx$ is a de Branges space of entire functions and $\ip{\wt{f}}{\wt{g}} _E =
\ip{\wt{f}}{\wt{g}} _{\wt{\mc{H}}}$ for all $\wt{f}, \wt{g} \in \wt{\mc{H}} =: \mc{H} (E)$.

Now let $r(z) := h(z) \hat{u} (z) = h(z) \ip{u}{\psi (\ov{z})}$. By Lemma \ref{lemming:nopole}, $\ip{u}{\psi(x)} \neq 0$ for any $x \in \bm{R}$, and it follows that $r$ has no zeroes or poles on $\bm{R}$ (the simple zeroes of $h$ on $\bm{R}$ coincide with the simple poles of $\hat{u}$). Hence for any $f \in \mc{H}$, $\wt{f} = r \hat{f}$, so that for any $f, g \in \mc{H}$,
\be \ip{f}{g} = \intfty \wt{f} (x) \ov{\wt{g} (x)} \frac{1}{|E(x)| ^2} dx = \intfty \hat{f} (x) \ov{\hat{g} (x)} \left| \frac{r(x)}{E(x)} \right| ^2 dx .\ee

The following theorem of Krein then implies that this measure $\sigma$ defined by $d\sigma (x) := \left| \frac{r(x)}{E(x)} \right| ^2 dx$ is in fact
a $u$-spectral measure for $T$ \cite[Theorem 2.1.1, pg. 49]{Krein}.

\begin{thm}{ (Krein)}
A Borel measure $\nu$ on $\bm{R}$ is a $u-$spectral measure if and only if $\ip{f}{g} = \intfty \hat{f} (x) \ov{\hat{g} (x)} d\nu (x)$ for all
$f, g \in \mc{H}$. \label{thm:specmeas}
\end{thm}

Note that since $E(x)$ has no real zeroes and $r$ has no real zeroes or poles, that $\sigma$ is in fact equivalent to Lebesgue measure on $\bm{R}$, and that $\sigma ' , \frac{1}{\sigma '}$ are both locally $L^\infty$.

The following theorem on $u-$spectral measures (the form below is valid for $T \in \symr{\mc{H}}$, and for our choice of gauge $u \in \ker{T^* -i}$) is also due to Krein \cite[Corollary 2.1 ,pg 16]{Krein}:

\begin{thm}{ (Krein)}
    Suppose that $T \in \symr{\mc{H}}$, and $0 \neq u \in \ker{T^* -i}$. Let $Q$ be a POVM of some densely defined self-adjoint extension $T' \supset T$, and let $\nu (\cdot )
:= \ip{ Q (\cdot ) u }{u}$ be a $u-$spectral measure of $T$. Then for any Borel set $\Om$,
\be \ip{ Q(\Om ) f }{g} = \int _\Om \hat{f} (x) \ov{\hat{g} (x)} d\nu (x) .\ee \label{thm:equalip}
\end{thm}

\subsubsection{Remark} Krein's theorems, Theorem \ref{thm:transform}, Theorem \ref{thm:specmeas} and Theorem \ref{thm:equalip}, were originally stated for densely defined $T \in \sym{\mc{H}}$  \cite{Krein}.
However, the extended statements above hold for non-densely defined $T$ with essentially no modification of Krein's original proofs.

Now suppose that $S \subset L^2 (\bm{R})$ and that $T = M_S \in \symr{S}$ is a restriction of $M$. Then $M$ is a self-adjoint extension of $M_S$, so that we can define the $u$-spectral measure $\mu (\Om ) := \ip{ \chi _\Om (M) u }{u}$. Since $M$ is multiplication by $x$ in $L^2 (\bm{R})$, the measure $\mu $ is absolutely continuous with respect to Lebesgue measure so that $d\mu (x) = \mu ' (x) dx$. Hence if $\ip{\hat{f}}{\hat{g}} _\mu := \intfty \hat{f} (x) \ov{\hat{g} (x)} \mu ' (x) dx$, then $\ip{\hat{f}}{\hat{g}} _\mu = \ip{f}{g}$ by Theorem \ref{thm:specmeas}.

Moreover, Theorem \ref{thm:equalip} implies that for any $f , g \in S$,

\ba \ip{ \chi _\Om (M) f }{g} & = & \int _\Om \hat{f} (x) \ov{\hat{g} (x) } \mu' (x) dx \\
&  = & \int _\Om \left| \frac{E(x)}{r(x)} \right| ^2 \mu ' (x) \wt{f} (x) \ov{\wt{g} (x)} \frac{1}{|E (x) | ^2} dx  \\
& = & \ip{ R(\wt{M} ) \chi _\Om (\wt{M} ) \wt{f}}{\wt{g}} _E , \label{eq:twiner}\ea
where $R(x) := \left| \frac{E(x)}{r(x)} \right| ^2 \mu ' (x) $ is locally $L^1$.  Here $\wt{M}$ denotes multiplication by the independent variable in $L^2 (\bm{R} , |E(x)| ^{-2} dx) \supset \mc{H} (E) = \wt{\mc{H}}$.

\subsubsection{Remark} \label{subsubsection:cyclic} In fact $\mu' (x) >0 \ \mr{a.e.}$. Otherwise there would be a Borel subset $\Om \subset \bm{R}$
of non-zero Lebesgue measure such that $\ip{\chi _\Om (\hat{M}) \hat{f} }{\hat{g}} _\mu =0$ for all $f,g \in \mc{H}$, where $\hat{M}$ denotes multiplication by $z$ in $\hat{\mc{H}} \subset L^2 (\bm{R} , d\mu )$. But this would imply that \be \left< \left| \frac{E(\wt{M} )}{r (\wt{M})} \chi _\Om (\wt{M}) \right| ^2  \wt{f} , \wt{g} \right> _E =0,  \ee for all $\wt{f}, \wt{g} \in \mc{H} (E)$, where $\wt{M}$ denotes multiplication by $z$ in $\mc{H} (E)$.
Since $E(x) / r(x) $ is non-zero almost everywhere with respect to Lebesgue measure, this would imply that elements of $\mc{H} (E)$ vanish almost everywhere
on $\Om$. This is impossible as elements of $\mc{H} (E)$ are entire functions. In conclusion $\mu ' >0$ almost everywhere. The fact that $\mu' >0$ almost everywhere where $\mu (\Om ) = \ip{ \chi _\Om (M) u }{u}$ also shows that the gauge $u$ is non-zero almost everywhere. This shows that the subspace $S$ contains an element which is non-zero almost everywhere with respect to Lebesgue measure so that $S$ is cyclic (and separating) for the von Neumann algebra generated by bounded functions of $M$.  The fact that $\mu ' >0$ almost everywhere also implies that $R(x) > 0 \ \mr{a.e.}$. These facts will be useful later.

Observe that
\ba \ip{R(\wt{M}) \wt{f} }{\wt{g} } _E & = & \intfty \frac{\wt{f} (x)}{r(x)} \ov{ \frac{\wt{g} (x)}{r(x)} } \mu ' (x) dx \\
& =& \ip{\hat{f}}{\hat{g}} _\mu  = \ip{f}{g} = \ip{\wt{f}}{\wt{g}} _E . \label{eq:partiso} \ea
This calculation shows that $R ^{1/2} (\wt{M}) P_E $ is a partial isometry in $L^2 (\bm{R} , |E(x) |^{-2} dx)$ with initial space $\mc{H} (E)$.

Now let $\theta := \frac{E^*}{E}$, a meromorphic inner function. Then multiplication by $\frac{1}{E}$ is an isometry of $L^2 (\bm{R} , | E(x) | ^{-2} dx)$
onto $L^2 (\bm{R} )$ that takes $\mc{H} (E)$ onto $K^2 _\theta$, and which intertwines $\wt{M}$ and $M$, the operators of multiplication by
the independent variable in $L^2 (\bm{R} , |E(x) | ^{-2} dx )$ and $L^2 (\bm{R} )$. Let $V: S \rightarrow K^2 _\theta $ be the isometry defined by
$Vf := \frac{\wt{f}}{E}$, and let $V_0 := V P_S$ be the corresponding partial isometry on $L^2 (\bm{R} )$. It then follows from equation (\ref{eq:twiner}) that given
any Borel set $\Om$ and $f,g \in L^2 (\bm{R})$,

\be \ip{P_S \chi _\Om (M) P_S f}{g} = \ip{ P_\theta R(M) \chi _\Om (M) P _\theta V_0 f}{V_0 g} .\ee

Let $\vnm$ denote the von Neumann algebra of $L^\infty$ functions of $M$, and let $R:= R(M) \geq 0 $, which is affiliated with $\vnm$. It follows that for any $m \in \vnm$.

\be  P_S m P_S = V_0 ^* P_\theta  \sqrt{R} \, m \, \sqrt{R} P _\theta V_0 . \ee

Given a projector $P$, we let $\mc{P}$ denote the completely positive map $\mc{P} (A) = P A P$, and if
$B \in B (L^2 (\bm{R}) )$, let $\mr{Ad} _B$ denote the completely positive map $\mr{Ad} _B (A) = B A B^*$.
The above equation shows that

\be \ad{V_0 ^*} \circ \mc{P} _\theta \circ \ad{\sqrt{R}}  \left| _{\vnm} \right. = \mc{P} _S \left| _{\vnm} \right. . \label{eq:cpeq} \ee

Note that since, by equation (\ref{eq:partiso}), $R^{1/2}  P _\theta$ is a partial isometry, that the completely positive map $\Phi _1:= \mc{P} _\theta \circ \ad{\sqrt{R}} : B (L^2 (\bm{R} ) ) \rightarrow B( K^2 _\theta ) $ is unital.

In the next section we will use the dilation theory of completely
positive maps to show that equation (\ref{eq:cpeq}) implies that the partial isometry $V_0 ^* : K^2 _\theta \rightarrow S$ acts as the restriction of an element affiliated with $\vnm$ to $S$, \emph{i.e.} $V_0 ^*$ acts as multiplication by a function $\ov{v(x)}$.  It will follow easily from this that $S$ is nearly invariant.

\section{Application of Dilation Theory}

It will be convenient to use a number of acronyms. CP means completely positive, CPU means CP and unital, TP means trace preserving. A CPTPU map is a completely positive unital and trace preserving map, which is also
sometimes called a quantum channel. SSD stands for Stinespring dilation.

The following lemma can be proven using Stinespring's theorem.

\begin{lemming}
Let $\mc{A}$ be a unital $C^*$ algebra. Suppose that $\phi _1 : \mc{A} \rightarrow B (\mc{H} _1 )$ and $\phi _2 : \ran{\phi_1} \rightarrow B (\mc{H} _2 )$
are CP maps such that $\mc{H} _i$ are separable. If $\pi _1$ and $\pi_2$ are the minimal Stinespring dilations of the $\Phi _1 = \phi _1$ and
$\Phi _2 := \phi _2 \circ \phi _1$, then there is a contractive $*$-homomorphism $\pi $ such that $\pi \circ \pi _1 = \pi _2$.
\end{lemming}

One can prove this by inspecting the proof of Stinespring's theorem as presented in \cite{Paulsen}.

\begin{proof}
Begin by constructing the representations $\pi _i$ as in the proof of Stinespring's theorem.  Consider the algebraic tensor products $\mc{A} \otimes \mc{H} _i =: \mc{K} _i '$.
Then define inner products on the $\mc{K} _i '$ by $( a \otimes x_i , b \otimes y_i) _i = \ip{\Phi _i (b^* a ) x_i}{ y_i } _i$ where $a,b \in \mc{A}$, $x_i, y_i \in \mc{H} _i$. Then
as per the usual proof, the Cauchy-Schwarz inequality can be applied to show that $\mc{N} _i := \{ u \in \mc{K} _i | \ (u , u) _i =0 \}$ is a vector subspace of $\mc{K} _i$. One
then defines the Hilbert spaces $\mc{K} _i$ to be the completions of $\mc{K} _i ' / \mc{N} _i$ with respect to the inner product $\ip{ u_i + \mc{N} _i}{ v_i + \mc{N} _i } _i := ( u _i , v_i ) _i$.
Now for $a \in \mc{A}$ define $\pi _i (a) : \mc{K} _i \rightarrow \mc{K} _i$ by $\pi _i (a) \sum a_k \otimes x_k = \sum a a_k \otimes x_k$. The usual proof of Stinespring's theorem shows that
this yields (not necessarily minimal) Stinespring dilations of the CP maps $\Phi _i$.

Now,
\ba \| \pi _1 (a)  \| & = & \sup _{ u = \sum a_j \otimes x_j + \mc{N} _1  \in \mc{K} _1 / \mc{N} _1  \ \ \| u \| _1 = 1} \left( \pi _1 (a) \sum a_j \otimes x_j , \pi _1 (a) \sum a_j \otimes x_j \right) _1 \nonumber \\
& = & \sum \ip{ \Phi _1 ( a_i ^* a^* a a_j ) x_j }{x_i} _{\mc{H} _1} \ea

It follows that if $\pi _1 (a) = 0 $ that for any $(a_1, ..., a_N) \in \mc{A} ^{(N)} = \otimes _{i =1} ^N \mc{A}$, and any $\vec{x} = (x_1, ... x_N) \in \mc{H} _1 ^{(N)}$ that
$\ip{ \Phi _1 ^{(N)} ([a_i ^* a^* a a_j ] ) \vec{x} }{\vec{x}} _{\mc{H} _1 ^{(N)}} =0$ so that $[a_i ^* a^* a a_j ] \in \ker{\Phi _1 ^{(N)}}$. Hence
$[a_i ^* a^* a a_j ] \in \ker \Phi _2 ^{(N)} = \phi _2 ^{(N)} \circ \Phi _1 ^{(N)}$, which in turn shows that $\| \pi _2 (a) \| =0$. Hence $\ker{\pi _1 } \subset \ker{\pi _2}$.

Define $\pi : \pi _1 (\mc{A} ) \rightarrow \pi _2 (\mc{A} )$ by $\pi \circ \pi _1 = \pi _2$. The above calculation shows that $\pi $ is a well-defined $*-$homomorphism. Also
$\pi _1 (a) \in \ker{\pi}$ if and only if $a \in \ker{\pi} _2 \supset \ker{\pi} _1$. Hence $\ker{\pi}$ is closed and is isomorphic to $\frac{\ker{\pi} _2}{\ker{\pi} _1}$. If
we define the map $\hat{\pi} : \pi _1 (\mc{A} / \ker{\pi} ) \rightarrow \pi _2 (\mc{A})$ then this is an isomorphism of $C^*$ algebras and is hence isometric. It follows that $\pi$ is
a contractive $*-$homomorphism.

\end{proof}

This basic fact will now be used to prove the following lemma:

\begin{lemming}
Let $\mc{B} \subset \mc{A}$ be $C^*$-algebras. Let $\Phi _i$ be CP maps from $\mc{A}$ into $B (\mc{H} _i)$. Let $\Phi  : B (\mc{H} _1 ) \rightarrow
B (\mc{H} _2 ) $ be a CPU map such that $\Phi \circ \Phi _1 | _{\mc{B}} = \Phi _2 | _{\mc{B}}$. Further assume that $\Phi _i$ and $\Phi _i | _\mc{B}$ have the
same minimal Stinespring dilations. Let $(\pi _i , V_i,  \mc{K} _i) $ be the minimal SSD's of the $\Phi _i$, $( \pi ' , V',  \mc{K} ') $ the minimal SSD of $\Phi \circ \Phi _1$. Then $\mc{K} _2 \subset \mc{K} '$
is reducing for $\pi ' | _\mc{B}$ and there is an onto
$*$-homomorphism $\pi :  \pi _1 (\mc{A} ) \rightarrow  \pi ' ( \mc{A} ) $ such that $\pi \circ \pi _1 | _\mc{B} = \pi _2 | _\mc{B} = \mc{P} _{\mc{K} _2} \circ \pi ' | _\mc{B}$.

\label{lemming:structure}

\end{lemming}

\begin{proof}

If $(\pi ' , V', \mc{K} ')$ is the minimal SSD of $\Phi \circ \Phi _1$, then it is automatically an SSD of $\Phi \circ \Phi _1 | _\mc{B} = \Phi _2 | _{\mc{B}}$. Since $\Phi_2$ and its
restriction to $\mc{B}$ have the same minimal SSD $(\pi _2 , V_2, \mc{K} _2 )$ it follows that we can assume $\mc{K} _2 \subset \mc{K} '$, that $\mc{K} _2 $
is reducing for $\pi ' | _{\mc{B}}$ and that $\mc{P} _{\mc{K} _2} \circ \pi ' | _{\mc{B}} = \pi _2 | _{\mc{B}} $. By the previous lemma,  there is an onto $^*$-homomorphism $\pi : \pi _1 (\mc{A} )
\rightarrow  \pi ' ( \mc{A} ) $ such that $\pi \circ \pi _1 = \pi '$. Hence $\pi \circ \pi _1 | _\mc{B} =
 \mc{P} _{\mc{K} _2} \circ \pi ' | _\mc{B} = \pi _2 | _\mc{B}$.
\end{proof}

Define $\Theta := \mc{P} _{\mc{K} _2} \circ \pi '$.  This is a CPU map which is a contractive $^*$-homomorphism when restricted to $\mc{B}$.

\begin{lemming}
If $S\subset L^2 (\bm{R})$ contains a function which is cyclic and separating for $\mr{vN} (M)$, \emph{i.e.} a
function $f $ which is non-zero almost everywhere with respect to Lebesgue measure, and $P$ is the projection onto $S$, then the minimal
SSD of $\mc{P} : \mr{vN} (M) \subset B(L^2 (\bm{R} ) ) \rightarrow B(S)$ is the identity map on $B(L^2 (\bm{R}))$. \label{lemming:easySSD}
\end{lemming}

Here, as before $\mc{P} (A) = P A P$ for any $A \in B(L^2 (\bm{R}))$.

\begin{proof}
    Straightforward: the identity map on $B(L^2 (\bm{R}))$ is clearly an SSD of $\mc{P} | _{\mr{vN} (M)}$. To show that it is minimal one just needs to check that $\mr{vN} (M) S$ is dense
in $L^2 (\bm{R})$.
As $S$ contains an element which is cyclic for $M$, this is clear.
\end{proof}

Applying this to our specific situation yields:

\begin{prop}
    Suppose that $S_i \subset L^2 (\bm{R})$ are cyclic (and hence separating) for $\mr{vN} (M)$ with projections $P_i$, and that there exists a CPU map $\Phi _1 : B(L^2 (\bm{R})) \rightarrow B(S_1)$
with minimal SSD $(\mr{id} , V , L^2 (\bm{R}) )$ for some contraction $V: B(S_1) \rightarrow B(L^2 (\bm{R}))$.   If there exists a CPU map $\Phi : B(S_1) \rightarrow B(S_2)$
such that $\Phi \circ \Phi _1 |_{\vnm} = \mc{P} _2 | _{\vnm}$, then there is a CPTPU map $\Theta : B(L^2 (\bm{R})) \rightarrow B(L^2(\bm{R}))$, such that
$\Theta (m) = m $ for all $m \in \mr{vN} (M)$ so that the effects of $\Theta$ belong to $\mr{vN} (M)$ and $ \mc{P} _2 \circ \Theta = \Phi \circ \Phi _1$. \label{prop:dil}
\end{prop}

Recall here that any completely positive map $\Phi : B(\mc{H} ) \rightarrow B(\mc{H})$ can be expressed as $\Phi (A) = \sum _i E_i A E_i ^* $ where the $E_i$ are contractions in $B (\mc{H})$ and
$\sum E_i E_i ^* \leq \bm{1}$.
If $\Phi$ is unital then it follows that $\sum E_i E_i ^* = \bm{1}$. These operators are called the effects of $\Phi $, or sometimes the Kraus operators of $\Phi$ and we write $\Phi \equiv \{ E_i \}$.
The set of effects of $\Phi$ is not unique, but two different sets of effects for $\Phi$ are related as described in Lemma \ref{lemming:effects} below.

\begin{proof}
Let $\mc{H} := L^2 (\bm{R})$.
We apply Lemma \ref{lemming:structure} with $\Phi _2 = \mc{P} _2$. By Lemma \ref{lemming:easySSD} the minimal SSD of $\mc{P} _2$ is $(\mr{id}, P_2 , L^2 (\bm{R} ))$. By Lemma \ref{lemming:structure}, there is a $*$-isomorphism $\pi : B(L^2 (\bm{R} )) \rightarrow B(L^2 (\bm{R} )) $ such that $\pi \circ \mr{id} = \pi '$, where $\pi '$ is the minimal SSD of $\Phi \circ \Phi _1$, and  $\pi |  _{\vnm} = \pi \circ \id | _{\vnm} = \id | _{\vnm}$.
Hence $\Theta := \mc{P} _{L^2 (\bm{R})} \circ \pi ' = \mc{P} _{L^2 (\bm{R})} \circ \pi \circ \mr{id} $ is a CPU map ($\pi _1 $ is the identity map)
$\Theta : B(L^2 (\bm{R} )) \rightarrow B(L^2 (\bm{R}))$ and we have that $\Theta | _{\vnm} = \pi _2 | _{\mr{vN} (M)} = \mr{id} | _{\mr{vN} (M)}$.

 In other words $\Theta (m) = m$ for all $m \in \vnm$ and hence if $\{ E_i \}$ are the effects of $\Theta$, then the $E_i $ commute with spectral projections of $M$ and must belong to $\vnm$ (this is not hard to show,
see \cite[pgs 7-8]{Beny}). In particular the effects of $\Theta$ are normal operators. Such a CP map is called hermitian.  Given a completely positive map $\Phi$ on $B(\mc{H})$, one can define its dual $\Phi ^\dag : T (\mc{H} ) \rightarrow T (\mc{H} )$,
with respect to the canonical trace on $B(\mc{H})$ by $\Phi ^\dag (T) \in T(\mc{H})$ is the unique trace-class operator obeying $\mr{Tr} (T \Phi (A) ) = \mr{Tr} ( \Phi ^\dag (T) A )$ for all
 $A \in B (\mc{H})$. Here $T (\mc{H})$ denotes the trace-class operators.  It is easy to show that $\Phi $ is unital if and only if $\Phi ^\dag$ is trace-preserving, and vice versa.
Since $\Theta$ is hermitian, it follows that $\Theta ^\dag$ is also unital. It follows that $\Theta$ is trace-preserving
and unital, hence $\Theta$ is a CPTPU map, \emph{i.e.} a quantum channel of $B(L^2 (\bm{R}))$.

Now \be \mc{P} _2 \circ \Theta = \mc{P} _2 \circ \pi \circ \pi _1 = \mc{P} _2 \circ \pi ' = \Phi \circ \Phi _1 , \ee and this completes the proof.

\end{proof}

We will need the following fact which relates two different sets of effects which define the same CP map acting on $B(\mc{H})$ when $\mc{H}$ is separable.

\begin{lemming}
Let $\Phi : B(\mc{H} ) \rightarrow B(\mc{H})$ be a normal CPU map and let $(E_i) _{i=1} ^k $ and $(F_j) _{j=1} ^l$ be two sets of effects for $\Phi$. Then there is an isometry $U : l^2 _k (\mc{H} ) \rightarrow l^2 _l (\mc{H})$
whose entries are scalars multiplied by the identity in $\mc{H}$ such that $U \, (E _i ^* ) = (F _j ^*)$ where here $(E_i ^*)$ denotes the column vector  with entries $E_i ^*$. In particular the two sets of effects have
the same linear span. \label{lemming:effects}
\end{lemming}

\begin{proof}
    In finite dimensions this is well-known to experts in quantum error correction, and the proof for the separable case is virtually identical. Here we sketch the proof.

    Let $(\mc{K} , V , \pi)$ denote the minimal SSD of $\Phi$ so that $V: \mc{H} \rightarrow \mc{K}$ is an isometry such that $V \pi (A) V^* = \Phi (A)$.  Since $\Phi$ is normal it follows
that $\pi$ is normal. Also since $\pi$ is a minimal SSD of $\Phi$, it is an irreducible normal representation of the type I factor $B(\mc{H})$.

It follows from the representation theory of factors of type I that we can assume that $\mc{K} = l^2 _k (\mc{H} ) \simeq \mc{H} \otimes l^2 _k $ for some $k \in \bm{N} \cup \{ \infty \}$ where $l^2 _k$ is the Hilbert space of square summable sequences of length
$k$, and that $\pi (A) = A \otimes \bm{1} $. Since $V: \mc{H} \rightarrow l^2 _k (\mc{H})$ we can define $E_k ^*  : \mc{H} \rightarrow \mc{H}$ by choosing $E_k ^* h = h_k$ where $Vh = ( h_1 , h_2 , ...)$.
The $(E_k)$ are a set of effects for $\Phi$, \emph{i.e.} $ \Phi (A) = \sum _k E_k A E_k ^* $, $\| E_k \| \leq 1 $ and $\sum _k E_k E_k ^* = \bm{1}$.

Now suppose that $(F_j)_{j=1} ^n$ are another set of effects for $\Phi$. Then we can construct a SSD of $\Phi$ by letting $\pi ' (A) = A \otimes \bm{1}$ on $l^2 _n (\mc{H})=: \mc{K} '$
and defining $V ' : \mc{H} \rightarrow \mc{K} ' $ by $V' h = (F_1 ^* h , F_2 ^* h , ...)$.  Now $(\mc{K} ' , V' , \pi ')$ contains a minimal SSD $(\mc{K} _2 , V' , \pi _2)$ (when constructing the minimal SSD
from an arbitrary SSD, this does not change the isometry $V'$, this can be observed from \cite[pg. 46]{Paulsen}) such that $\pi ' (B (\mc{H})) V' \mc{H} = \mc{K} _2$.

By the uniqueness of the minimal SSD, there is a unitary operator $U: \mc{K} = l^2 _j (\mc{H}) \rightarrow \mc{K} _2 \subset l^2 _n (\mc{H})$ such that $\mr{Ad} _U \circ \pi = \pi _2$ and $UV = V'$.
The first equation implies that if we write $U$ as an $n \times j$ matrix with entries in $B(\mc{H})$, then each entry $U_{ik}$ belongs to the commutant of $B(\mc{H} )$ and hence must be a scalar times
the identity. The second equation tells us that this scalar matrix multiplying the column vector $(E_i ^*)$ equals the column vector $(F_j ^*)$.  In particular the $(E_i)$ and $(F_i)$ have the same linear span.
\end{proof}

To apply the result of the previous proposition to the situation of the previous section, equation (\ref{eq:cpeq}), we will need one final lemma:

\begin{lemming}
    Consider $\Phi _1 := \mc{P} _\theta \circ \ad{\sqrt{R}} : B(L^2 (\bm{R} )) \rightarrow B (K^2 _\theta)$. Then the minimal SSD's of both $\Phi _1 $ and $\Phi _1 | _{\vnm}$ are both equal to $(\mr{id} ,  \sqrt{R} P _\theta , L^2 (\bm{R}) )$, where $\mr{id}$ denotes the identity isomorphism.
\end{lemming}

\begin{proof}
   Recall that $V= \sqrt{R} P _\theta : K^2 _\theta \rightarrow L^2 (\bm{R})$ is an isometry. For any $A \in B (L^2 (\bm{R} )$, we have that $V^* \mr{id} (A) V = \mc{P} _\theta \circ \ad{\sqrt{R}} (A) = \Phi _1 (A)$, this shows that $\mr{id}$ is a SSD of $\Phi _1$, and hence
of $\Phi _1 | _{\vnm}$.  To show that this is minimal we need to show that both $B(L^2 (\bm{R} ) ) V K^2 _\theta$ and $\vnm V K^2 _\theta$ are dense in $L^2 (\bm{R} )$. Clearly the
first set is dense in $L^2 (\bm{R})$.  Now it is not difficult to show that $L^2 (\bm{R} ) = \sum _{k \in \bm{Z} } \theta ^k K^2 _\theta$. Since $\sqrt{R}$ is non-zero almost everywhere with respect
to Lebesgue measure, it follows that $\vnm  V K^2 _\theta $ is dense in $L^2 (\bm{R} )$.
\end{proof}

This next corollary is the main result of this paper:

\begin{cor}
If $S, K^2 _\theta$ are the subspaces of the previous section, then $S = u h K^2 _{\theta ' } $ is nearly invariant, where $u$ is unimodular,  $\theta ' := \frac{\theta (i) - \theta }{1- \ov{\theta(i)} \theta}$ is such that $\theta ' (i) =0$, and $h$ is an isometric multiplier of $K^2 _{\theta ' }$ onto $\ov{u}S$ (so that $\frac{h}{z+i} \in \ov{u}S$), and $\theta ' $ is the Lisvic characteristic function of $M_S$. \label{cor:main}
\end{cor}

\begin{proof}
  Let $S_1 := K^2 _\theta$, $S_2 = S$, with projectors $P_i$. Let $\Phi _1 := \mc{P} _1 \circ \ad{\sqrt{R}} $, $\Phi _2 = \mc{P} _2 $ and $\Phi = \ad{V_0 ^*}$. Then by equation (\ref{eq:cpeq}) of the previous
section, the previous lemma, and Remark \ref{subsubsection:cyclic}, it follows that the conditions of Proposition \ref{prop:dil} are satisfied, so that there is a quantum channel $\Theta $ on $B(L^2 (\bm{R}))$ with effects $\{E_i \} \subset \vnm$ and
$\mc{P} _2 \circ \Theta = \Phi \circ \Phi _1$. Taking adjoints yields $\Theta ^\dagger \circ \mc{P} _2 = \Phi _1 ^\dagger \circ \Phi ^\dagger$. Hence both $\{ E_i ^* P_2 \}$ and
$\{ \sqrt{R} P_1 V_0 \}$ are sets of effects for the same map, and so by Lemma \ref{lemming:effects}, they must have the same linear span. This shows that for any $i$, there is an $\alpha _i \in \bm{C}$
so that $E_i ^* P_2 = \alpha _i \sqrt{R} P_1 V_0 P_2 $ (recall $V_0 : S_2 \rightarrow S_1$ is a partial isometry). Hence,
\be \left( E_i ^* - \frac{\alpha _i}{\alpha _1} E_1 ^* \right) P_2 = 0 , \ee
 and since $S = S_2$ is cyclic and separating for $\vnm$, we conclude that $E_i ^* = \frac{\alpha _i}{\alpha _1} E_1 ^*$.
Since $\Theta $ is unital, we have $1 = \sum  |c_i | ^2 | |E_1 (x) | ^2 =: k^2 |E_1 (x) | ^2$. This shows that $U := k E_1$ is a unimodular function such that $\Theta = \ad{U}$, so that $\Theta$
is actually a $*$-isomorphism. Now $\{ U P_2 \}$ and $\{ \sqrt{R} P_1 V_0 \}$ have the same linear span, and there is an $\alpha \in \bm{C}$ so that
\be \alpha U P_2 = \sqrt{R} P_1 V_0 = \sqrt{R} V_0.\ee Hence $V_0 = \frac{\alpha U}{\sqrt{R}} P_2 $. Actually, since $\frac{U}{ \sqrt{R}} P_2$ and $V_0$ are both partial isometries, it follows that we can take $\alpha =1$.
This shows that multiplication by the function $  U  / \sqrt{R}$ is an isometry from $S$ onto $K^2 _\theta$.
Hence multiplication by $ \ov{U} \sqrt{R}$ is an isometry from $K^2 _\theta $ onto $S$. Also by known results there is a function $h$ such that multiplication by $h$ is an isometry from $K^2 _{\theta '}$ onto
$K^2 _\theta$, this mapping is called a Crofoot transform \cite[Section 13]{Sarason-alg}. It follows that if $g :=  h' \ov{U} \sqrt{R}$, that multiplication by $g$ is an isometry from $K^2 _{\theta '}$ onto $S$.
Since $\theta' (i) =0$, $k_i (z) = \frac{i}{2\pi} \frac{1}{z+i}$ is the point evaluation vector at $i$ in $K^2 _{\theta '}$, it follows that $\frac{g}{z+i} \in L^2 (\bm{R})$. It follows that $S = g K^2 _{\theta '}$ is nearly invariant, and if $ u h $ is the Beurling-Nevanlinna factorization of $\frac{g}{z+i}$, $h \in H^2$, $u$ is unimodular, that $S' = \ov{u} S $ is a nearly invariant subspace of $H^2$ such that $S' = h (z+i) K^2 _{\theta ' }$.
Since $M_S$ is unitarily equivalent to $M_{\theta '}$, it follows that the characteristic function of $M_S$ is $\theta '$.

\end{proof}

\begin{cor}
    If $R=1$, then $S$ is seminvariant. \label{cor:unital}
\end{cor}

Here $S \subset L^2 (\bm{R})$ is called seminvariant if it is seminvariant for the shift (multiplication by $\mu (x) = \frac{x-i}{x+i}$). Recall that
a subspace is seminvariant for an operator if it is the direct difference of two invariant subpsaces, one of which contains the other. A subspace
is seminvariant for the shift if and only if $S = u K^2 _\theta$ where $u$ is unimodular and $\theta$ is an inner function.  This follows
from the Beurling-Lax theorem, see for example the proof of \cite[Theorem 5.2.2]{Martin-dB}.

\begin{proof}
    Suppose that $R=1$. In this case $UP_2 = V_0$ (we can assume $\alpha =1$), so that $U^* P_1 = V_0 ^*$ and $S = S_2 = U^* K^2 _\theta$ where $U^* \in \vnm$ is unitary.

\end{proof}

It seems possible that the converse to the above corollary is also true.

\begin{cor}
    If $S \subset L^2 (\bm{R} )$ is such that $M$ has a restriction $M_S \in \symr{S}$, then $S$ is a reproducing kernel Hilbert space with a $\bm{T}$-parameter family of total orthogonal sets of point evaluation vectors.
\end{cor}

\begin{proof}
    This follows as $S$ is the image of $K^2 _\theta$ under an isometric multiplier and $K^2 _\theta$ has these properties when $\theta$ is inner and meromorphic.

\end{proof}

\section{Outlook}

We have proven that a subspace $S \subset L^2 (\bm{R})$ is nearly invariant with $S = h K^2 _\theta$, and $\theta $ meromorphic and inner, $\theta (i) =0$, if and only if the multiplication operator $M$ has
a restriction $M_S \in \symr{S}$ with meromorphic inner characteristic function $\theta$. We expect a similar result to hold whenever $\theta$ is inner and not necessarily meromorphic, and perhaps an analogous result could be established for arbitrary contractive analytic $\theta$. However to generalize the approach presented here would require generalizing Krein's results of Section \ref{section:reptheory} to the case of more general contractive analytic functions.

\end{document}